\documentclass[11pt,letterpaper]{amsart}

\usepackage[english]{babel}
\usepackage{bm}

\usepackage{stmaryrd}
\usepackage{upgreek}

\usepackage{color}

\usepackage[utf8]{inputenc}
\usepackage{hyperref}
\usepackage{relsize}
\usepackage{amssymb,amsmath}
\usepackage{dsfont}
\usepackage[all]{xy}
\usepackage{mathtools}
\usepackage{mathrsfs}

\renewcommand{\theta}{\uptheta}
\renewcommand{\iota}{\upiota}
\renewcommand{\alpha}{\upalpha}
\renewcommand{\beta}{\upbeta}
\renewcommand{\gamma}{\upgamma}
\renewcommand{\delta}{\updelta}
\renewcommand{\zeta}{\upzeta}
\renewcommand{\pi}{\uppi\hspace{0.05em}}
\renewcommand{\xi}{\upxi}
\renewcommand{\chi}{\upchi}
\renewcommand{\sigma}{\upsigma}
\renewcommand{\Lambda}{\Uplambda}
\renewcommand{\Gamma}{\Upgamma}
\renewcommand{\phi}{\upphi}
\renewcommand{\nu}{\upnu}
\renewcommand{\tau}{\uptau}
\renewcommand{\mu}{\upmu}
\renewcommand{\eta}{\upeta}

\newtheorem{theorem}{Theorem}[section]

\newtheorem{proposition}[theorem]{Proposition}
\newtheorem{lemma}[theorem]{Lemma}

\theoremstyle{definition}

\theoremstyle{remark}

\newtheorem{remark}[theorem]{Remark}

\newcommand{\dd}{\mathbf{d}}

\newcommand{\ee}{\mathbf{e}}

\DeclareMathOperator{\sst}{-ss}

\DeclareMathOperator{\Hom}{Hom}

\DeclareMathOperator{\Exp}{Exp}

\DeclareMathOperator{\Gl}{GL}

\DeclareMathOperator{\Jac}{Jac}

\DeclareMathOperator{\Tr}{Tr}

\DeclareMathOperator{\pt}{pt}

\DeclareMathOperator{\HO}{\mathbf{H}}

\DeclareMathOperator{\Coha}{\mathcal{H}}

\DeclareMathOperator{\vmult}{\mathbf{m}}
\DeclareMathOperator{\QT}{\mathcal{A}}

\newcommand{\BoA}{\mathbb{A}}
\newcommand{\BoC}{\mathbb{C}}
\newcommand{\BoN}{\mathbb{N}}
\newcommand{\BoP}{\mathbb{P}}
\newcommand{\BoQ}{\mathbb{Q}}
\newcommand{\BoZ}{\mathbb{Z}}

\newcommand{\CM}{\mathcal{M}}
\newcommand{\CN}{\mathcal{N}}
\newcommand{\CS}{\mathcal{S}}
\newcommand{\CZ}{\mathcal{Z}}

\newcommand{\FM}{\mathfrak{M}}
\newcommand{\FS}{\mathfrak{S}}

\title[The generic Markov CoHA is not spherically generated]{The generic Markov CoHA is not spherically generated}
\author{Ben Davison}

\begin{document}
\maketitle
\begin{abstract}
Let $Q$ be the Markov quiver, and let $W$ be an infinitely mutable potential for $Q$.  We calculate some low degree refined BPS invariants for the resulting Jacobi algebra, and use them to show that the critical cohomological Hall algebra $\Coha_{Q,W}$ is not necessarily spherically generated, and is not independent of the choice of infinitely mutable potential $W$.  This leads to a counterexample to a conjecture of Gaiotto, Grygoryev and Li \cite[\S 2.1]{GGL}, but also suggestions for how to modify it.  In the case of generic cubic $W$, we discuss a way to modify the conjecture, by excluding the non-spherical part via the decomposition of $\Coha_{Q,W}$ according to the characters of a discrete symmetry group.
\end{abstract}
\section{Preliminaries}
Given a quiver $Q$ with potential $W\in\BoC Q/[\BoC Q,\BoC Q]$ the Kontsevich--Soibelman cohomological Hall algebra (CoHA) $\Coha_{Q,W}$ is an associative algebra which provides a beautiful link between two worlds (see \cite{KS2} for details).  On the one hand, taking partition functions encoding the dimensions of the graded pieces of $\Coha_{Q,W}$ and factorizing them according to the slopes determined by a given stability condition, we may extract the \textit{refined BPS invariants} of the category of representations for the Jacobi algebra $\Jac(Q,W)$ associated to $Q$ and $W$.  These invariants have their origins in physics, and should be thought of as counting BPS states on the noncommutative Calabi--Yau threefold associated to $\Jac(Q,W)$.

On the other hand, $\Coha_{Q,W}$ is an \textit{algebra} and for suitable choices of $Q$ and $W$, this algebra can be shown to recover and extend various quantum groups, and may be used to prove new results regarding Yangian-type algebras \cite{botta2023okounkov}.

On the algebraic side, the quivers $Q$ with potential for which the algebra $\Coha_{Q,W}$ has been most intensively studied are \textit{symmetric}, meaning that for every pair of vertices $i$ and $j$ in $Q$, there are as many arrows from $i$ to $j$ as there are from $j$ to $i$.  From the point of view of studying BPS invariants, this is quite a restrictive set of quivers: it is the set of quivers for which the BPS invariants are independent of stability conditions, and all wall-crossing phenomena disappear.  Also, from the point of view of cluster algebras, the class of symmetric quivers is an unnatural choice, since in that subject (see \cite{DWZ1} for background) the usual restriction on quivers is that they contain no loops or 2-cycles.  A symmetric quiver satisfying these restrictions has no arrows at all!

This short paper is inspired by a pair of related conjectures in \cite[\S 2.1]{GGL}.  The first states that if $W$ is an infinitely mutable\footnote{This is a kind of non-degeneracy condition arising in cluster algebras.  Rather than spell out the definition, we refer to Lemma \ref{germ_lemma} for examples and non-examples of infinitely mutable potentials.} potential for a quiver $Q$ containing no loops or 2-cycles, then $\Coha_{Q,W}\cong \CS_Q$, where $\CS_Q$ is the spherical subalgebra of the shuffle algebra $\Coha_{Q}$ (see \S \ref{sec_3} for partial definitions).  Note that, while calculations inside $\Coha_{Q,W}$ are made rather difficult by the necessity of working with vanishing cycle cohomology, the algebra $\Coha_Q$ has a very down-to-earth presentation, and it may be studied and understood, along with its subalgebra $\CS_Q$, using elementary calculations and computer algebra packages.  So it would be excellent news to discover that the algebra $\Coha_{Q,W}$, for which it is hard to calculate products, and for which the graded dimensions recover refined BPS invariants, is in fact isomorphic to $\CS_Q$.  The weaker version of this conjecture, also stated in \cite[\S 2.1]{GGL} states that $\Coha_{Q,W}$ is independent of $W$, as long as $W$ is chosen to be infinitely mutable.
\smallbreak 
The \textit{Markov quiver}, for which the definition is recalled in \S \ref{Markov_sec}, has a well-established reputation as a source of interesting properties, examples, and counterexamples in the theory of cluster algebras; see \cite{DWZ1,ANC} and references therein.  The study of the cluster algebra built from this quiver is closely connected to the study of solutions to Markov's equation; we refer to \cite{LLRS} for recent work in this direction, along with further references.  True to its reputation, in this short paper we present counterexamples to the above conjectures (Propositions \ref{main_prop} and \ref{ce_2}), built from the Markov quiver with infinitely mutable potentials.  More positively, we will see that the Markov quiver provides example calculations that suggest how the conjecture might be modified.

\subsection{Setup}
By a quiver $Q$ we mean a finite directed graph.  We set $Q_0$ to be the set of vertices of $Q$, $Q_1$ to be the set of arrows, and $s,t\colon Q_1\rightarrow Q_0$ to be the two morphisms sending an arrow to its source and target, respectively.  Let $\dd\in \BoN^{Q_0}$ be a dimension vector.  We denote by $\FM_{\dd}(Q)$ the stack of $\dd$-dimensional $\BoC Q$-modules.  It has dimension $-\chi_Q(\dd,\dd)$, where $\chi_Q$ is the \textit{Euler form} defined by
\begin{align*}
\chi_Q\colon &\BoN^{Q_0}\times\BoN^{Q_0}\rightarrow \BoZ\\
&(\dd,\ee)\mapsto \sum_{i\in Q_0} \dd_i\ee_i-\sum_{a\in Q_1}\dd_{s(a)}\ee_{t(a)}.
\end{align*}
We can present the stack $\FM_{\dd}(Q)$ as a global quotient stack, as we briefly recall.  We define $\BoA_{\dd}(Q)\coloneqq \prod_{a\in Q_1}\Hom(\BoC^{\dd_{s(a)}},\BoC^{\dd_{t(a)}})$, a vector space parameterising $\dd$-dimensional $\BoC Q$-modules, which we may consider as an affine variety in the obvious way.  This is acted on by the gauge group $\Gl_{\dd}\coloneqq \prod_{i\in Q_0} \Gl_{\dd_i}(\BoC)$ by simultaneous change of basis.  Then $\FM_{\dd}(Q)\cong \BoA_{\dd}(Q)/\Gl_{\dd}$.  

Let $W\in \BoC Q/[\BoC Q,\BoC Q]$ be a linear combination of cyclic paths.  Taking the trace of $W$, considered as an endomorphism of the underlying vector spaces of $\BoC Q$-modules, provides a function $\Tr(W)$ on $\FM_{\dd}(Q)$.

The Kontsevich--Soibelman critical CoHA $\Coha_{Q,W}$ is a $\BoN^{Q_0}$-graded associative algebra, for which the underlying vector space of the $\dd$th graded piece is the vanishing cycle cohomology
\[
\Coha_{Q,W,\dd}=\HO(\FM_{\dd}(Q),\phi_{\Tr(W)}\BoQ[\chi_Q(\dd,\dd)])
\]
and the square brackets denote the cohomological shift.  In words, the $\dd$th graded piece of $\Coha_{Q,W}$ is the hypercohomology of the perverse sheaf of vanishing cycles for the function $\Tr(W)$ on the stack of $\dd$-dimensional $Q$-representations.  

The associative product $\vmult\colon \Coha_{Q,W}^{\otimes 2}\rightarrow \Coha_{Q,W}$ is defined in \cite[\S 7]{KS2}.  We remark that the product respects the cohomological grading on $\Coha_{Q,W}$ if and only if $Q$ is symmetric.  In general the failure of the CoHA multiplication to preserve the cohomological grading is captured by the following formula relating cohomological degrees, where we assume $\alpha\in\Coha_{Q,W,\dd}$ and $\beta\in\Coha_{Q,W,\ee}$, and we use $\circ$ to denote the CoHA multiplication:
\begin{equation}
\label{cojump}
\lvert \alpha\circ \beta\lvert =\lvert\alpha\lvert +\lvert\beta\lvert+\chi_{Q}(\dd,\ee)-\chi_Q(\ee,\dd).
\end{equation}

Fix a quiver $Q$.  We define the ring $\QT_Q$ as follows.  It is a $\BoZ(\!(q^{1/2})\!)$-module, and as a $\BoZ(\!(q^{1/2})\!)$-module it is equal to the set of formal linear combinations $\sum_{\dd\in\BoN^{Q_0}}a_{\dd}(q^{1/2})x^{\dd}$ with each $a_{\dd}(q^{1/2})\in \BoZ(\!(q^{1/2})\!)$.  The multiplication is given by extending the rule $x^{\dd}x^{\ee}=(-q^{1/2})^{\chi_Q(\dd,\ee)/2}x^{\dd+\ee}$ to formal linear combinations.

We consider the partition function in $\QT_Q$
\[
\CZ_{Q,W}(x)\coloneqq\sum_{\dd\in\BoN^{Q_0}} \chi_{q^{1/2}}\left(\Coha_{Q,W,\dd}\right)x^{\dd}
\]
where for a $\BoZ$-graded vector space $V$ we set
\[
\chi_{q^{1/2}}(V)=\sum_{n\in\BoZ}\dim(V^n)(-q^{1/2})^n.
\]
\begin{remark}
\label{naive_remark}
Conceptually, it often makes more sense to replace the above Poincar\'e series with a ``weight'' Poincar\'e series that is sensitive to the mixed Hodge structure on $\Coha_{Q,W}$, and in particular the weight filtration.  See \cite[\S 7]{KS2} for definitions and details of this approach.  Since in this paper we will only be interested in calculating graded dimensions of certain vector spaces, we ignore this alternative, and instead take naive Poincar\'e series throughout.
\end{remark}


Let $\zeta\in\BoQ^{Q_0}$ be a \textit{stability condition}.  We define the \textit{slope} $\mu(\dd)$ of a dimension vector $\dd\in\BoN^{Q_0}\setminus\{0\}$ by setting 
\[
\mu(\dd)=\frac{\dd\cdot \zeta}{\sum_{i\in Q_0}\dd_i}.
\]
We assume that $\zeta$ is \textit{generic}, meaning that if $\dd,\ee\in\BoN^{Q_0}\setminus \{0\}$ have the same slope, then $\chi_Q(\dd,\ee)=\chi_Q(\ee,\dd)$.  Given a slope $\theta\in (-\infty,\infty)$ we define 
\[
\Lambda_{\theta}^{\zeta}\coloneqq \{\dd\in\BoN^{Q_0}\;\lvert \; \dd=0\textrm{ or } \mu(\dd)=\theta\}.
\]
We define $\QT_{Q,\theta}$ to be the subring of $\QT_Q$ spanned by formal $\BoZ(\!(q^{1/2})\!)$-linear combinations of symbols $x^{\dd}$ where $\dd\in\Lambda_{\theta}^{\zeta}$.  By genericity of $\zeta$, for every $\theta$ the ring $\QT_{Q,\theta}$ is commutative.  There is a unique factorization
\begin{equation}
\label{fact_id}
\CZ_{Q,W}(x)=\prod_{\infty\xrightarrow{\theta}-\infty}\CZ^{\zeta}_{Q,W,\theta}(x)
\end{equation}
where $\CZ^{\zeta}_{Q,W,\theta}(x)\in \QT_{Q,\theta}$.  By the cohomological wall crossing isomorphism \cite[Thm.B]{QEAs} there are equalities
\begin{equation}
\label{WCF}
\CZ^{\zeta}_{Q,W,\theta}(x)=\sum_{\dd\in\Lambda_{\theta}^{\zeta}}\chi_{q^{1/2}}(\HO(\FM^{\zeta\sst}_{\dd}(Q),\phi_{\Tr(W)}\BoQ[-\chi_Q(\dd,\dd)]))x^{\dd}
\end{equation}
where $\FM^{\zeta\sst}_{\dd}(Q)\subset \FM_{\dd}(Q)$ is the substack of $\zeta$-semistable $Q$-representations.

We may repackage the functions $\CZ^{\zeta}_{Q,W,\theta}(x)$ in terms of \textit{refined BPS invariants}, which are Laurent polynomials $\Omega_{Q,W,\dd}^{\zeta}\in\BoZ[q^{\pm 1/2}]$ defined via the equality
\begin{equation}
\label{ICF}
\CZ^{\zeta}_{Q,W,\theta}(x)=\Exp\left(\sum_{\dd\in\Lambda_{\theta}^{\zeta}\setminus \{0\}} \Omega^{\zeta}_{Q,W,\dd}x^{\dd}(-q^{1/2})(1-q)^{-1}\right).
\end{equation}
Here $\Exp$ is the plethystic exponential, defined by setting
\[
\Exp\left(\sum_{\substack{\dd\in\Lambda_{\theta}^{\zeta}\setminus \{0\}\\n\in\BoZ}}a_{\dd,n}x^{\dd}q^{n/2}\right)\coloneqq \prod_{\substack{\dd\in\Lambda_{\theta}^{\zeta}\setminus \{0\}\\n\in\BoZ}}(1-q^{n/2}x^{\dd})^{-a_{\dd,n}}
\]
whenever the right hand makes sense.  The fact that the formal power series $\Omega_{Q,W,\dd}^{\zeta}\in\BoZ(\!(q^{1/2})\!)$ defined this way are actually Laurent polynomials is a consequence of the cohomological integrality theorem \cite[Thm.A]{QEAs}.
\begin{remark}
The polynomials $\Omega^{\zeta}_{Q,W,\dd}$ can be realised by taking the Poincar\'e polynomials of BPS cohomology, introduced in \cite{QEAs}.  If we had defined the partition functions $\CZ^{\zeta}_{Q,W,\theta}(x)$ using weight series instead, we would take the weight polynomials of BPS cohomology to recover the corresponding refined BPS invariants.  Since in this paper we are principally interested in the dimensions of vector spaces, it is most natural to consider naive Poincar\'e series.
\end{remark}
\subsection{The conjectures}
Fixing a quiver $Q$, it is very interesting to study the dependence of $\CZ_{Q,W}(x)$ on $W$.  It is conjectured in \cite[\S 2.1]{GGL} that as long as the quiver with potential $(Q,W)$ is \textit{infinitely mutable}, the partition function $\CZ_{Q,W}(x)$ does not depend on the choice of $W$.  This is equivalent to the statement that after fixing a stability condition $\zeta$, the BPS invariants $\Omega_{Q,W,\dd}^{\zeta}$ do not depend on $W$.  Being infinitely mutable is a certain non-degeneracy condition on quivers with potentials that is important in the categorification of cluster algebras via Ginzburg's differential graded algebras (see e.g. \cite{ginz,MR2681708, DWZ1} for definitions, motivation, and background).  It is, first of all, assumed that $Q$ does not contain loops and 2-cycles.  Then mutation at a given vertex $i$ produces a new quiver with potential $\mu_i(Q,W)$.  Infinite mutability is the condition that this mutated quiver also does not contain 2-cycles, and that this remains the case after iterated mutation at any sequence of vertices.
\section{The Markov quiver}

\label{Markov_sec}
\subsection{Potentials for the Markov quiver}
For the rest of the paper we fix $Q$ to be the Markov quiver.  Precisely, we set $Q_0=\{1,2,3\}$ and $Q_1=\{a_1,a_2,b_1,b_2,c_1,c_2\}$, with the orientations of the arrows as in the following diagram
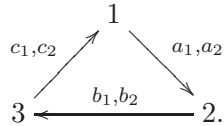
\[
\xymatrix{
&1\ar[dr]^{a_1,a_2}\\3\ar[ur]^{c_1,c_2}&&2.\ar[ll]_{b_1,b_2}
}
\]
Let $W\in \BoC Q/\BoC Q,\BoC Q]$ be a potential.  We grade $\BoC Q$ by path length. 
 \begin{lemma}\cite[Chapter 14, Example 4.5]{MR1264417}\label{germ_lemma}
After applying a graded linear isomorphism $\Phi\colon \BoC Q\rightarrow \BoC Q$, i.e. an isomorphism taking arrows to linear combinations of arrows, we may write $W$ in one of the following five forms
\begin{enumerate}
\item $W=W_{\geq 6}$ \label{vbad_case}
\item $W= c_1b_1a_1+W_{\geq 6}$\label{bad_case}
\item $W=c_1b_1a_1+c_1b_2a_2 +W_{\geq 6}$ \label{bad_case2}
\item $W=c_1b_1a_2+c_1b_2a_1+c_2b_1a_1+W_{\geq 6}$\label{marginal_case}
\item $W=c_1b_1a_1+c_2b_2a_2+W_{\geq 6}$, \label{generic_case}
\end{enumerate}
where $W_{\geq 6}$ is the sum of all of the homogeneous pieces of $W$ of degree at least $6$, i.e. a linear combination of cyclic paths of length at least $6$.   

Moreover, case \eqref{generic_case} is \emph{generic}, in the following sense: the type of a potential $W$ under graded linear isomorphisms is determined by the cubic part of $W$, and a generic homogeneous cubic potential $W$ can be transformed to the form $W=c_1b_1a_1+c_2b_2a_2$. 
\end{lemma}

\begin{proposition}
\label{IM_pots}
The potential $W$ is infinitely mutable if and only if it is of one of the forms given in cases \eqref{marginal_case} and \eqref{generic_case} above.
\end{proposition}
\begin{proof}
In the first three cases, a single mutation at vertex $2$ produces a quiver with potential that contains at least two $2$-cycles; see \cite{DWZ1} for the definition of mutation for quivers with potentials.  So we just need to show that in the remaining two cases, the quiver with potential is infinitely mutable.  For case \eqref{generic_case}, this is \cite[Example.8.6]{DWZ1}.  The argument for case \eqref{marginal_case} is the same as the argument for \eqref{generic_case}; we write it for completeness.

We describe the mutation of $(Q,W)$ at vertex $2$; this will suffice, since the quiver with potential is invariant under rotational symmetry by the group $\BoZ/3\BoZ$.  The new quiver has arrows $a_1^*,a_2^*$ from $2$ to $1$, arrows $b_1^*,b_2^*$ from $3$ to $2$, arrows $c_1,c_2$ from $3$ to $1$, and four arrows $[b_ia_j]$ from $1$ to $3$, with $i,j\in\{1,2\}$.  The new potential, before cancelling off quadratic terms, is of the form
\[
W'=c_1[b_1a_2]+c_1[b_2a_1]+c_2[b_1a_1]+\sum_{i,j\in\{1,2\}}[b_ia_j]a_j^*b_i^*+W'_{\geq 4}.
\]
Writing $u=[b_1a_2]+[b_2a_1]$ and $v=[b_1a_2]-[b_2a_1]$ we find
\[
W'=c_1u+c_2[b_1a_1]+[b_1a_1]a_1^*b_1^*+[b_2a_2]a_2^*b_2^*+(u+v)a_2^*b_1^*/2+(u-v)a_1^*b_2^*/2+W'_{\geq 4}
\]
Substituting $c_1\mapsto c_1-a_2^*b_1^*/2-a_1^*b_2^*/2+\ldots$ and $c_2\mapsto c_2-a_1^*b_1^*+\ldots$ and rescaling $v$ and $a_1^*$ this potential transforms to
\[
W''=c_1u+c_2[b_1a_1]+[b_2a_2]a_2^*b_2^*+va_2^*b_1^*+va_1^*b_2^*+W''_{\geq 6}
\]
where $W''_{\geq 6}$ does not contain the arrows $c_1,u,c_2,[b_1a_1]$.  Removing the quadratic terms and the 2-cycles $c_1u$ and $c_2[b_1a_1]$ we find that the mutated quiver with potential contains no $2$-cycles, is isomorphic to the Markov quiver, and the new potential is again of the form \eqref{marginal_case}.
\end{proof}

\subsection{BPS invariants for generic $W$ and small dimension vectors}
\label{low_BPS}
For the rest of the paper we fix a stability condition $\zeta\in \BoQ^{Q_0}$ by setting $\zeta_1=1$, $0<\zeta_2=\epsilon \ll 1$ and $\zeta_3=-1$.  

Next, we calculate some low-degree refined BPS invariants for potentials of generic form (\eqref{generic_case} above).  Setting $\dd=d\delta_i$ to be the dimension vector that is zero everywhere apart from $i\in Q_0$ and for which the entry at $i$ is $d$, we find $\FM^{\zeta\sst}_{\dd}(Q)\cong \pt/\Gl_d(\BoC)$ and the function $\Tr(W)$ is zero on this stack.  So if $\theta=1,\epsilon, -1$, we have the standard calculation
\begin{align*}
\CZ^{\zeta}_{Q,W,\theta}(x)=&\sum_{d\geq 0}\chi_{q^{1/2}}(\HO(\pt/\Gl_d(\BoC),\BoQ[-d^2]))x^{d\delta_i}\\
=&\Exp\left(x^{\delta_i}\frac{-q^{1/2}}{(1-q)}\right).
\end{align*}
In particular, \[
\Omega^{\zeta}_{Q,W,\delta_i}=1\]
for $i=1,2,3$.

Now let $\dd=(1,1,0)$.  A $\dd$-dimensional $Q$-representation is given by two linear maps $\rho(a_1)\colon \BoC\rightarrow \BoC$ and $\rho(a_2)\colon \BoC\rightarrow \BoC$, satisfying the condition that not both of them are the zero map.  We thus see that $\FM^{\zeta\sst}_{\dd}(Q)\cong \BoP^1/\BoC^*$.  Again, the function $\Tr(W)$ is zero on this stack, and we have the isomorphism of sheaves $\phi_{\Tr(W)}\BoQ[-\chi_Q(\dd,\dd)]\cong \BoQ$.  Comparing \eqref{WCF} and \eqref{ICF} we deduce
\[
\Omega^{\zeta}_{Q,W,(1,1,0)}=-q^{-1/2}-q^{1/2},
\]
which is the normalized Poincar\'e polynomial of $\BoP^1$.

On the other hand, there are no $\zeta$-semistable $Q$-representations of dimension vector $(1,0,1)$; such a module $\rho$ would have a destabilising submodule of dimension vector $(1,0,0)$.  So it follows, again from \eqref{WCF}, that
\[
\Omega^{\zeta}_{Q,W,(1,0,1)}=0.
\]
\setcounter{tocdepth}{1}
\begin{proposition}
\label{rbpscalc}
Continue to assume that $W$ is generic, i.e. that we can write $W=c_1b_1a_1+c_2b_2a_2+W_{\geq 6}$.  Then \begin{itemize}
\item
$\Omega^{\zeta}_{Q,W,(1,1,1)}=2+e(q^{1/2})$, where $e(q^{1/2})\in\BoN[(-q^{1/2})^{\pm 1}]$ is a Laurent polynomial in $q^{1/2}$, with the coefficient of $q^{n/2}$ positive or negative depending on whether $n$ is even or odd.  
\item
If we set $W=c_1b_1a_1+c_2b_2a_2$ then $\Omega^{\zeta}_{Q,W,(1,1,1)}=2$.
\end{itemize}
\end{proposition}
\begin{proof}
Let $\rho$ be a $\zeta$-semistable $(1,1,1)$-dimensional $Q$-representation.  Fixing identifications between the vector spaces that $\rho$ assigns to the three vertices and the one-dimensional vector space $\BoC$, $\rho$ is determined by six linear maps $\rho(a_1),\rho(a_2),\ldots$, which we may identify with numbers in $\BoC$.  We abuse notation by denoting these numbers $a_1,\ldots,c_2$.  Then stability for $\rho$ is equivalent to the two conditions
\begin{itemize}
\item
At least one of $a_1,a_2$ are nonzero.
\item
At least one of $b_1,b_2$ are nonzero.
\end{itemize}
Let $\CM=\CM^{\zeta\sst}_{(1,1,1)}(Q)$ be the coarse moduli space; since $(1,1,1)$ is indivisible, this is a fine moduli space, and moreover we have $\FM^{\zeta\sst}_{(1,1,1)}(Q)\cong \CM/\BoC^*$.

We cover $\CM$ by the four charts $\CM_{i,j}$, for $i,j\in\{1,2\}$, where $\CM_{i,j}$ is defined to be the subvariety corresponding to $Q$-representations for which $a_i\neq 0$ and $b_j\neq 0$.  Then each of $\CM_{i,j}$ is isomorphic to $\BoA^4$; up to gauge equivalence $a_i=1,b_j=1$, and then the remaining 4 arrows provide the four coordinates of affine 4-space.  
We prove the final part of the proposition first, so for now we set $W=c_1b_1a_1+c_2b_2a_2$.  I claim that
\[
\phi_{\Tr(W)}\BoQ[4]\vert_{\CM_{i,j}}=\begin{cases} 0& \textrm{if }i=j\\
\BoQ_{0} &\textrm{if }i\neq j.
\end{cases}
\]
The first case ($i=j$) is easy: in local coordinates we write
\[
\Tr(W)=c_i+c_kb_ka_{k}
\]
with $k\neq i$.  In particular, the critical locus of this function is empty, and since $\phi_{\Tr(W)}\BoQ[3]$ is supported on this locus, the first part of the claim follows.  

For the second case ($i\neq j$), we have instead
\begin{equation}
\label{spc}
\Tr(W)=c_ja_i+c_ib_j.
\end{equation}
By the Thom--Sebastiani isomorphism \cite{Ma01}, we find 
\[
\phi_{\Tr(W)}\BoQ_{\CM_{i,j}}[4]\cong \phi_{c_ja_i}\BoQ_{\BoA^2}[2]\boxtimes \phi_{c_ia_j}\BoQ_{\BoA^2}[2]\cong \BoQ_0\boxtimes \BoQ_0.
\]
We have used here the standard calculation $\phi_{xy}\BoQ_{\BoA^2}[2]\cong \BoQ_0$, the constant sheaf supported on the origin $0\in\BoA^2$.

Let $\alpha\in\CM$ be the point corresponding to the module for which $a_1$ and $b_2$ act via isomorphisms, and all other arrows act via the zero map.  Let  $\beta\in\CM$ be the point corresponding to the module for which $a_2$ and $b_1$ act via isomorphisms, and all other arrows act via the zero map.  We depict them as follows
\begin{equation}
\label{alpha_beta}
\xymatrix@R=1em{\alpha\colon &1\ar[r]^{a_1}&2\ar[r]^{b_2}&3\\
\beta\colon &1\ar[r]^{a_2}&2\ar[r]^{b_1}&3.
}
\end{equation}
The claim tells us that $\phi_{\Tr(W)}\BoQ_{\CM}[4]\cong \BoQ_{\alpha}\oplus\BoQ_{\beta}$.

Then we have
\begin{align*}
\HO(\FM^{\zeta\sst}_{(1,1,1)}(Q),\phi_{\Tr(W)}\BoQ[3])\cong&\HO(\CM^{\zeta\sst}_{(1,1,1)}(Q),\phi_{\Tr(W)}\BoQ[4])\otimes \HO(\pt/\BoC^*,\BoQ[-1])\\
\cong&\HO(\{\alpha,\beta\},\BoQ)\otimes \HO(\pt/\BoC^*,\BoQ[-1])
\end{align*}
and so
\begin{align*}
\chi_{q^{1/2}}(\HO(\FM^{\zeta\sst}_{(1,1,1)}(Q),\phi_{\Tr(W)}\BoQ[3]))x^{(1,1,1)}=2\cdot \frac{-q^{1/2}}{1-q}.
\end{align*}
Now we consider the case of general $W=c_1b_1a_1+c_2b_2a_2+W_{\geq 6}$.  In this case we find that the scheme-theoretic critical locus of $\Tr(W)$ contains the points $\alpha$ and $\beta$ as reduced connected components, since after a formal change of coordinates we may transform $\Tr(W)=c_ja_i+a_ib_j+(\textrm{higher order terms})$ back into the form \eqref{spc}.  It follows that the restriction of $\phi_{\Tr(W)}\BoQ_{\CM}[4]$ to a small analytic neighbourhood of $\alpha$ is $\BoQ_{\alpha}$, and its restriction to a small analytic neighbourhood of $\beta$ is $\BoQ_{\beta}$.  Set $\CN=\CM\setminus\{\alpha,\beta\}$.  Passing to derived global sections, we find that there is a direct sum decomposition
\[
\HO(\CM^{\zeta\sst}_{(1,1,1)}(Q),\phi_{\Tr(W)}\BoQ[4])\cong \HO(\{\alpha,\beta\},\BoQ)\oplus \HO(\CN,\phi_{\Tr(W)}\BoQ[4]).
\]
Then we set $e(q^{1/2})=\chi_{q^{1/2}}(\HO(\CN,\phi_{\Tr(W)}\BoQ[4]))$.
\end{proof}
\section{Counterexamples}
\label{sec_3}
\subsection{Spherical (non) generation}
We refer to \cite[\S 2]{KS2} for the definition of the shuffle algebra $\Coha_Q$ associated to an arbitrary quiver.  It is shown there that this shuffle algebra is isomorphic to the cohomological Hall algebra $\Coha_{Q,W}$ with $W=0$.  At the level of underlying vector spaces, we have 
\[
\Coha_{Q,\dd}\cong\BoQ[z_{1,1},\ldots,z_{1,\dd_1},z_{2,1},\ldots,z_{l,1},\ldots,z_{l,\dd_l}]^{\FS_{\dd}}
\]
where the symmetric group $\FS_{\dd}=\prod_{i\in Q_0} \FS_{\dd_i}$ acts by permuting all variables while preserving the first of their two subscripts.  The cohomological grading of a homogeneous polynomial $p(z)$ is given by setting 
\[
\lvert p(z)\lvert=2\deg(p(z))+\chi_Q(\dd,\dd).
\]
We continue to denote by $Q$ the Markov quiver from \S \ref{Markov_sec}.  We define the \textit{spherical subalgebra} $\CS_Q\subset \Coha_Q$ to be the subalgebra generated by all the subspaces $\Coha_{Q,\delta_i}\subset \Coha_Q$ for $i\in Q_0$.  More generally, we define the spherical subalgebra $\CS_{Q,W}\subset \Coha_{Q,W}$ to be the subalgebra generated by the subspaces $\Coha_{Q,W,\delta_i}$ for $i\in Q_0$, and we say that $\Coha_{Q,W}$ is spherically generated if it is equal to its spherical subalgebra.

By the formula for the shuffle product in \cite[\S 2]{KS2}, $\CS_{Q,(1,1,1)}$ is spanned by elements of the form
\begin{equation}
\label{spherical_span}
(z_{1}-z_{3})^2p,\quad (z_{3}-z_{2})^2r,\quad (z_{2}-z_{1})^2s
\end{equation}
where we have abbreviated $z_i=z_{i,1}$ for $i=1,2,3$ and $p,r,s\in\BoZ[z_1,z_2,z_3]$.  In particular, we find that 
\[
\CS^n_{Q,(1,1,1)}\cong \begin{cases} 0 &\textrm{if } n<1\\
\BoQ\!\cdot\! (z_{1}-z_{3})^2\oplus \BoQ\!\cdot\! (z_{2}-z_{1})^2\oplus \BoQ\!\cdot\!(z_{3}-z_{2})^2&\textrm{if }n=1.\end{cases}
\]
Taking dimensions:
\[
\dim(\CS^n_{Q,(1,1,1)})= \begin{cases} 0 &\textrm{if } n<1\\
3&\textrm{if }n=1.\end{cases}
\]
If instead we allow nonzero potential $W$, we find that we still have isomorphisms $\Coha_{Q,W,\delta_i}\cong\BoQ[z_i]$, and via the cohomological degree calculation of \eqref{cojump}, the following lemma:
\begin{lemma}
\label{sg_lem}
For $Q$ the Markov quiver and $W$ arbitrary, we have
\[
\CS^n_{Q,W,(1,1,1)}=\begin{cases} 0&\textrm{if }n<1\\
\mathrm{Span}(z_1^0\circ z_3^0\circ z_2^0,\;z_2^0\circ z_1^0\circ z_3^0,\;z_3^0\circ z_2^0\circ z_1^0) &\textrm{if }n=1.
\end{cases}
\]
\end{lemma}

Now we reinstate the assumption that $W$ is generic.  From the calculations of BPS invariants in \S \ref{low_BPS} we calculate the dimensions of $\Coha_{Q,W,\dd}^n$ for low values of $n$:

\begin{align*}
\CZ_{Q,W}(x)=&\left(1+x^{(0,0,1)}\frac{-q^{1/2}}{1-q}\right)\ast\left(1+x^{(0,1,1)}(-q^{-1/2}-q^{1/2})\frac{-q^{1/2}}{1-q}\right)\ast\left(1+x^{(0,1,0)}\frac{-q^{1/2}}{1-q}\right)\ast\\
&\ast \left(1+(2+e(q^{1/2}))x^{(1,1,1)}\frac{-q^{1/2}}{1-q}\right)\ast\left(1+x^{(1,1,0)}(-q^{-1/2}-q^{1/2})\frac{-q^{1/2}}{1-q}\right)\ast\\ &\ast\left(1+x^{(1,0,0)}\frac{-q^{1/2}}{1-q}\right)+\textrm{higher order terms}\\
\end{align*}
where the higher order terms are linear combinations of monomials $x^{\dd}$ with at least one of $\dd_1,\dd_2,\dd_3>1$, and $e(q^{1/2})$ is the Laurent polynomial introduced in Proposition \ref{rbpscalc}.  Write $u(q^{1/2})x^{(1,1,1)}$ for the $x^{(1,1,1)}$ term of $\CZ_{Q,W}(x)$.  Then the above factorization of $\CZ_{Q,W}(x)$ yields
\begin{align*}
u(q^{1/2})x^{(1,1,1)}=-q^{1/2}&\big(\frac{q}{(1-q)^3}x^{(0,0,1)}\ast x^{(0,1,0)}\ast x^{(1,0,0)}+ \\&+\frac{1+q}{(1-q)^2}(x^{(0,0,1)}\ast x^{(1,1,0)}+x^{(0,1,1)}\ast x^{(1,0,0)})+\\&+ \frac{2+e(q)}{1-q}x^{(1,1,1)}\big)
\end{align*}
and so 
\begin{equation}
\label{coeff_calc}
u(q^{1/2})=-q^{1/2}\frac{2-q^2+(2+e(q))(1-q)^2}{(1-q)^3}.
\end{equation}
Observing that the coefficients of even powers of $q^{1/2}$ in $e(q^{1/2})$ are positive, we deduce that the $q^{1/2}$ coefficient of $u(q^{1/2})$ is at least $4$, and so, comparing with Lemma \ref{sg_lem} we deduce the following:
\begin{proposition}
\label{main_prop}
Let $Q$ be the Markov quiver from \S \ref{Markov_sec}.  Let $W=c_1b_1a_1+c_2b_2a_2+W_{\geq 6}$ be a generic potential.  Then $\Coha_{Q,W}$ is not spherically generated.  Moreover, $\dim(\Coha^1_{Q,W,(1,1,1)})>\dim(\CS_{Q,(1,1,1)}^1)$.
\end{proposition}
In particular, there is no (graded) isomorphism $\Coha_{Q,W}\cong \CS_Q$.
\subsection{Excluding non-spherical generators}
\label{corrected_conj}
Let $G=\BoZ/2\!\cdot\!\BoZ$.  Fix the potential $W=c_1b_1a_1+c_2b_2a_2$.  We consider the $G$-action on $Q$ that swaps $a_1$ with $a_2$, $b_1$ with $b_2$ and $c_1$ with $c_2$.  This action fixes $W$.  As such, $G$ acts on the critical cohomology $\Coha_{Q,W}$, and it is easy to see that the CoHA multiplication is $G$-equivariant.  Furthermore, $G$ acts trivially on $\Coha_{Q,W,\delta_i}$ for $i=1,2,3$, and so $G$ acts trivially on the \emph{entire} spherical subalgebra.  Therefore, letting $\Coha_{Q,W}^{\mathrm{sgn}}\subset \Coha_{Q,W}$ be the summand carrying the sign representation for $G$, elements of this summand are not spherically generated.

With $\alpha$ and $\beta$ the $(1,1,1)$-dimensional representations introduced in \eqref{alpha_beta}, the vector space $\phi_{\Tr(W)}\BoQ_{\CM^{\zeta\sst}_{(1,1,1)}(Q)}[4]\cong \BoQ_{\alpha}\oplus\BoQ_{\beta}$, which is the BPS cohomology giving rise to the BPS invariant $\Omega^{\zeta}_{Q,W,(1,1,1)}=2$, carries the regular $G$-representation.  The Poincar\'e series $\Omega^{\zeta,G\mathrm{-inv}}_{Q,W,(1,1,1)}$ of the $G$-invariant part of the BPS cohomology is thus $1$, and so repeating the calculation of \eqref{coeff_calc} we find the generating function for the $G$-invariant part of the CoHA:
\[
\chi_{q^{1/2}}(\Coha_{Q,W,(1,1,1)}^{G\mathrm{-inv}})=\frac{2-q^2+(1-q)^2}{(1-q)^3}.
\] 
On the other hand, from \eqref{spherical_span} we can calculate
\begin{align*}
\chi_{q^{1/2}}(\CS_{Q,(1,1,1)})=&-q^{-3/2}((1-q)^{-2}-1-2q)(1-q)^{-1}\\
=&\chi_{q^{1/2}}(\Coha_{Q,W,(1,1,1)}^{G\mathrm{-inv}}).
\end{align*}
Put differently, the non-spherically generated part of $\Coha_{Q,W,(1,1,1)}$ is given by elements $u^n\cdot (1_{\alpha}-1_{\beta})$, where $u\in\BoQ[u]=\HO(\pt/\BoC^*,\BoQ)$ acts via multiplication by the first Chern class of the determinant line bundle.  It is possible to show that the algebra generated by these elements surjects onto the free exterior algebra $\mathcal{A}$ generated by the same symbols.  A physically motivated possible modification of the spherical generation conjecture, suggested by Davide Gaiotto, is that $\Coha_{Q,W}$ splits as the product of $\mathcal{A}$ and $\CS_Q$.  Via dimensional reduction \cite[Appendix A]{Da13} and Proposition \ref{qh_pot} it should be possible to test this prediction for low dimension vectors.

\subsection{Dependence on $W$}
The second part of Proposition \ref{main_prop} provides a counterexample to the conjecture regarding spherical subalgebras  in \cite[\S 2.1]{GGL}.  A weaker conjecture, also stated in \cite[\S 2.1]{GGL}, is that for infinitely mutable $W$, $\Coha_{Q,W}$ is independent of $W$.  Comparing the two parts of Proposition \ref{rbpscalc}, this would imply that $e(q^{1/2})=0$ for all $W$.  To exclude the ``error term'' $e(q^{1/2})$ one could instead conjecture that $\Coha^{\textrm{nilp}}_{Q,W}$ is independent of $W$, where we define $\iota_{\dd}\colon \FM^{\textrm{nilp}}_{\dd}(Q)\hookrightarrow \FM_{\dd}(Q)$ to be the inclusion of the reduced substack containing the nilpotent representations, and 
\[
\Coha^{\textrm{nilp}}_{Q,W,\dd}=\HO(\FM^{\textrm{nilp}}_{\dd}(Q),\iota_{\dd}^!\phi_{\Tr(W)}\BoQ[\chi_Q(\dd,\dd)]).
\]
The multiplication is again as defined in \cite[\S 7]{KS2}.  Alternatively, one could conjecture that for quasi-homogeneous infinitely mutable potentials $W$, the CoHA $\Coha_{Q,W}$ is independent of $W$.  In this final section, on the one hand we show that the Markov quiver provides counterexamples to these forms of the independence conjecture, but on the other hand our results will indicate a way forward with a weakened version of the spherical generation conjecture.

We consider the ``marginal'' potential $W_{\mathtt{marg}}=c_1b_1a_2+c_1b_2a_1+c_2b_1a_1$ -- the homogeous potential of type \eqref{marginal_case}.  Recall from Proposition \ref{IM_pots} that this potential is infinitely mutable.
\begin{lemma}
\label{marg_calc}
There is an equality of generating series
\[
\chi_{q^{1/2}}(\Coha_{Q,W_{\mathtt{marg}},(1,1,1)})=-q^{1/2}\frac{3-2q}{(1-q)^3}.
\]
\end{lemma}
\begin{proof}
By Verdier self-duality of the vanishing cycle sheaf, we have the isomorphism
\[
\HO(\FM_{(1,1,1)}(Q),\phi_{\Tr(W_{\mathtt{marg}})}\BoQ[-\chi_Q(\dd,\dd)])\cong \HO_{\mathrm{c}}(\FM_{(1,1,1)}(Q),\phi_{\Tr(W_{\mathtt{marg}})}\BoQ[3])^{\vee}
\]
where the right hand side is the graded vector dual of the compactly supported hypercohomology.  Let $Q'$ be the quiver obtained from $Q$ by removing the arrows $c_1$ and $c_2$.  Define 
\begin{align*}
A\coloneqq&\BoC Q'/\langle \partial W_{\mathtt{marg}}/\partial c_1,\partial W_{\mathtt{marg}}/\partial c_2\rangle\\
\cong&\BoC Q'/\langle b_1a_2+b_2a_1,a_1b_1\rangle.
\end{align*}
We denote by $\FM_{(1,1,1)}(A)$ the stack of $(1,1,1)$-dimensional $A$-modules.  By the dimensional reduction isomorphism \cite[Appendix.A]{Da13} there is an isomorphism
\[
\HO_{\mathrm{c}}(\FM_{(1,1,1)}(Q),\phi_{\Tr(W_{\mathtt{marg}})}\BoQ)\cong \HO_{\mathrm{c}}(\FM_{(1,1,1)}(A),\BoQ)[-2].
\]
The stack $\FM_{(1,1,1)}(A)$ is isomorphic to the global quotient stack $Z/T$, where $Z\subset \BoA^4$ is cut out by the equations $b_1a_2+b_2a_1=0$ and $a_1b_1=0$, $T=(\BoC^*)^3$, and the first copy of $\BoC^*$ scales the $a_i$ coordinates, the second scales the $b_i$ coordinates, and the third acts trivially.  We define
\begin{align*}
U_1=&\{(a_1,a_2,b_1,b_2)\in Z\;\lvert\; a_1= 0,a_2=0\}\\
U_2=&\{(a_1,a_2,b_1,b_2)\in Z\;\lvert\; a_1= 0,a_2\neq 0\}\\
U_3=&\{(a_1,a_2,b_1,b_2)\in Z\;\lvert\; a_1\neq  0\}.
\end{align*}
Then
\begin{align*}
U_1/T&\cong \BoA^2/T\\
U_2/T&\cong \BoA^1/(\BoC^*)^2\\
U_3/T&\cong \BoA^1/(\BoC^*)^2.
\end{align*}
These three stacks stratify $\FM_{(1,1,1)}(A)$.  All of the above stacks have pure mixed Hodge structures on their compactly supported cohomology, so that their weight Poincar\'e series agree with their naive Poincar\'e series, and the above stratification gives the identity
\begin{align*}
\chi_{q^{1/2}}(\HO_{\mathrm{c}}(\FM_{(1,1,1)}(A),\BoQ)^{\vee})=&\sum_{1\leq i\leq 3} \chi_{q^{1/2}}(\HO_{\mathrm{c}}(U_i/T,\BoQ)^{\vee})\\
=&\left(\frac{q^2}{(q-1)^3}+\frac{2q}{(q-1)^2}\right)_{q\mapsto q^{-1}}
\end{align*}
as required
\end{proof}
Comparing with the analogous calculation for the generic infinitely mutable potential $W_{\mathtt{gen}}=c_1b_1a_1+c_2b_2a_2$, yields the following.
\begin{proposition}
\label{ce_2}
The CoHA $\Coha_{Q,W}$ is not independent of the choice of infinitely mutable potential $W$.  There are equalities of refined BPS invariants $\Omega^{\zeta}_{Q,W_{\mathtt{gen}},(1,1,1)}=3$ and $\Omega^{\zeta}_{Q,W_{\mathtt{marg}},(1,1,1)}=2$.
\end{proposition}
\begin{proof} 

Comparing Lemma \ref{marg_calc} with \eqref{coeff_calc} we find that the $x^{(1,1,1)}$ coefficient of $\mathcal{Z}_{Q,W_{\mathtt{gen}}}(x)-\mathcal{Z}_{Q,W_{\mathtt{marg}}}(x)$ is given by
\begin{equation}
\label{wedge_diff}
-q^{1/2}\frac{(4-4q+q^2)-(3-2q)}{(1-q)^3}=-q^{1/2}\frac{1}{1-q}.
\end{equation}
Since this difference is nonzero, the graded dimensions of $\Coha_{Q,W_{\mathtt{gen}},(1,1,1)}$ and $\Coha_{Q,W_{\mathtt{marg}},(1,1,1)}$ are not the same.

The equality $\Omega^{\zeta}_{Q,W_{\mathtt{gen}},(1,1,1)}=3$ is Proposition \ref{rbpscalc}.  Since for dimension vectors $\dd'$ strictly less than $(1,1,1)$ in the natural partial order, we have $\Omega^{\zeta}_{Q,W_{\mathtt{gen}},\dd'}=\Omega^{\zeta}_{Q,W_{\mathtt{marg}},\dd'}$, it follows from \eqref{fact_id} and \eqref{ICF} that \eqref{wedge_diff} is equal to $-q^{1/2}\frac{\Omega^{\zeta}_{Q,W_{\mathtt{gen}},(1,1,1)}-\Omega^{\zeta}_{Q,W_{\mathtt{marg}},(1,1,1)}}{1-q}$.
\end{proof}
By the above calculation, and Proposition \ref{qh_pot} below, $\Coha_{Q,W_{\mathtt{marg}}}$ \textit{is} spherically generated in degree $(1,1,1)$.  We note, following Lemma \ref{germ_lemma}, that this is a \textit{non generic} infinitely mutable potential.    Whether spherical generation continues for the marginal potential, for higher dimension vectors, is an interesting problem, that (via dimensional reduction) may again be tested numerically.  More generally, an interesting modification of the conjecture in \cite[\S 2.1]{GGL} would be that \textit{for every quiver there exists at least one infinitely mutable potential for which the CoHA is spherically generated}.  

We finish with a proposition which should be useful for studying this conjecture.  Before stating it we recall that a potential $W$ is called \textit{quasihomogeneous} if there is a grading $p\colon Q_1\rightarrow \BoN$ such that all the cycles appearing in $W$ are of the same total degree $d$ with respect to the grading $p$, and $d>0$.
\begin{proposition}
\label{qh_pot}
Let $Q$ be a quiver without loops, and let $W$ be a quasihomogeneous potential.  Then $\Coha_{Q,W}\cong \CS_Q$ as graded algebras if and only if $\CZ_{Q,W}(x)=\sum_{\dd\in\BoN^{Q_0}}\chi_{q^{1/2}}(\CS_{Q,\dd})x^{\dd}$.  In this case, the spherical subalgebra of $\Coha_{Q,W}$ is isomorphic to $\CS_Q$.
\end{proposition}
\begin{proof}
One implication is trivial: if two graded algebras are isomorphic, they certainly have the same graded dimensions.  So we need to show the reverse implication.  For this, we consider the morphism of CoHAs $\xi\colon \Coha_{Q,W}\rightarrow \Coha_Q$ constructed in \cite[Prop.4.4]{botta2023okounkov} (this uses that $W$ is quasihomogeneous).  Since $Q$ has no loops, it follows that $\Tr(W)=0$ when restricted to each of the stacks $\FM_{\delta_i}(Q)$ for $i\in Q_0$.  It follows that $\xi$ induces an isomorphism when we restrict to the $\delta_i$ graded piece.  In particular, the image of $\xi$ contains $\CS_Q$.  By the equality of graded dimensions, the image of $\xi$ is precisely $\CS_Q$, and $\xi$ induces an isomorphism $\Coha_{Q,W}\cong \CS_Q$.  The final statement follows, since $\CS_Q$ is spherically generated by definition.
\end{proof}
Via Proposition \ref{qh_pot}, the kinds of calculations of BPS invariants performed in this paper may be used to not just test the variants of the spherical generation conjecture discussed above, but try to prove them.

\subsection{Acknowledgements}
I am grateful to Wei Li for discussions regarding the conjectures in \cite{GGL}, and also to the organisers at the SwissMAP research station, for providing an exhilarating place to do maths.  Thanks also to Davide Gaiotto for illuminating correspondence regarding the CoHAs and potentials appearing in this paper.  This research was funded by a University Research Fellowship of the Royal Society (no. 221040).
\bibliographystyle{alpha}
\bibliography{Literatur}

\begin{thebibliography}{LLRS23}

\bibitem[BD23]{botta2023okounkov}
T.~M. Botta and B.~Davison.
\newblock {Okounkov's conjecture via BPS Lie algebras}.
\newblock {\em arXiv preprint arXiv:2312.14008}, 2023.

\bibitem[Dav17]{Da13}
B.~Davison.
\newblock {The critical CoHA of a quiver with potential}.
\newblock {\em Q. J. Math.}, 68(2):635--703, 2017.

\bibitem[DM20]{QEAs}
B.~Davison and S.~Meinhardt.
\newblock {Cohomological Donaldson--Thomas theory of a quiver with potential
  and quantum enveloping algebras}.
\newblock {\em Invent. Math.}, 221(3):777--871, 2020.

\bibitem[DWZ08]{DWZ1}
H.~Derksen, J.~Weyman, and A.~Zelevinsky.
\newblock Quivers with potentials and their representations. {I}. {M}utations.
\newblock {\em Selecta Math. (N.S.)}, 14(1):59--119, 2008.

\bibitem[GGL24]{GGL}
D.~Gaiotto, N.~Grygoryev, and W.~Li.
\newblock {Categories of Line Defects and Cohomological Hall Algebras}.
\newblock 2024.
\newblock arXiv:2406.07134.

\bibitem[Gin06]{ginz}
V.~Ginzburg.
\newblock {Calabi--Yau algebras}, 2006.
\newblock arXiv:0612139.

\bibitem[GKZ94]{MR1264417}
I.~M. Gelfand, M.~M. Kapranov, and A.~V. Zelevinsky.
\newblock {\em Discriminants, resultants, and multidimensional determinants}.
\newblock Mathematics: Theory \& Applications. Birkh\"auser Boston, Inc.,
  Boston, MA, 1994.

\bibitem[Kel10]{MR2681708}
B.~Keller.
\newblock Cluster algebras, quiver representations and triangulated categories.
\newblock In {\em Triangulated categories}, volume 375 of {\em London Math.
  Soc. Lecture Note Ser.}, pages 76--160. Cambridge Univ. Press, Cambridge,
  2010.

\bibitem[KS11]{KS2}
M.~Kontsevich and Y.~Soibelman.
\newblock {Cohomological Hall algebra, exponential Hodge structures and motivic
  Donaldson--Thomas invariants}.
\newblock {\em Commun. Number Theory Phys.}, 5, 2011.
\newblock arXiv:1006.2706.

\bibitem[LLRS23]{LLRS}
K.~Lee, L.~Li, M.~Rabideau, and R.~Schiffler.
\newblock On the ordering of the {M}arkov numbers.
\newblock {\em Adv. in Appl. Math.}, 143:Paper No. 102453, 29, 2023.

\bibitem[Mas01]{Ma01}
D.~Massey.
\newblock {The Sebastiani--Thom isomorphism in the Derived Category}.
\newblock {\em Comp. Math.}, 125(3):353--362, 2001.

\bibitem[NC12]{ANC}
A.~N\'ajera~Ch\'avez.
\newblock On the c-vectors and g-vectors of the {M}arkov cluster algebra.
\newblock {\em S\'em. Lothar. Combin.}, 69, 2012.

\end{thebibliography}

\vfill

\textsc{\small Ben Davison: School of Mathematics and Maxwell Institute for Mathematical Sciences, University of Edinburgh, United Kingdom}\\
\textit{\small E-mail address:} \texttt{\small ben.davison@ed.ac.uk}\\

\end{document}